\documentclass[oneside]{amsart}
\usepackage[latin1]{inputenc}
\usepackage[T1]{fontenc}
\usepackage{amssymb}
\usepackage{vmargin}
\usepackage{mathrsfs}
\usepackage[ps,dvips,arrow,matrix]{xy}
\usepackage{lmodern}

\setmarginsrb{2.5cm}{1.5cm}{2.5cm}{3.5cm}{1cm}{1cm}{1cm}{1cm}

\renewcommand{\leq}{\leqslant}
\renewcommand{\geq}{\geqslant}

\renewcommand{\P}{\mathbf{P}}
\newcommand{\F}{\mathbf{F}}
\newcommand{\Q}{\mathbf{Q}}
\newcommand{\C}{\mathbf{C}}
\newcommand{\Z}{\mathbf{Z}}
\newcommand{\Ql}{\Q_\ell}
\newcommand{\Zl}{\Z_\ell}
\newcommand{\kbarre}{{\overline{k}}}
\newcommand{\Krond}{\mathscr{K}}
\newcommand{\Erond}{\mathscr{E}}
\newcommand{\Orond}{\mathscr{O}}
\newcommand{\rk}{\mathrm{rank}}
\newcommand{\et}{{\text{ét}}}
\newcommand{\Gm}{{\mathbf{G}_\mathrm{m}}}
\newcommand{\tors}[2]{{\vphantom{#2}}_{#1}{#2}}
\newcommand{\isoto}{\xrightarrow{\sim}}
\newcommand{\abs}[1]{\left| \mskip1mu #1 \right|}

\DeclareMathOperator{\Pic}{Pic}
\DeclareMathOperator{\Spec}{Spec}
\DeclareMathOperator{\Br}{Br}
\DeclareMathOperator{\NS}{NS}
\DeclareMathOperator{\Div}{Div}

\newtheorem{theorem}{Theorem}[section]
\newtheorem{lemma}[theorem]{Lemma}
\newtheorem{prop}[theorem]{Proposition}
\newtheorem{remark}[theorem]{Remark}

\title{Transcendental Brauer-Manin obstruction on a pencil of elliptic curves}
\author{Olivier Wittenberg}
\date{April 8, 2003}
\address{UMR 8628, Mathématiques, Bâtiment 425, Université de Paris-Sud, F-91405 Orsay, France}
\email{olivier.wittenberg@ens.fr{\normalfont{} (or }olivier.wittenberg@normalesup.org{\normalfont{})}}

\begin{document}
\begin{abstract}
This note gives an explicit example of transcendental Brauer-Manin obstruction to weak approximation.  It has two features which the only previously known
example of such obstruction did not have: the class in the Brauer group which is responsible for the obstruction is divisible, and the underlying algebraic
variety is an elliptic surface.
\end{abstract}
\maketitle

\section{Introduction}

Let $\Br(X)$ denote the cohomological Brauer group $H^2_\et(X,\Gm)$ of a scheme $X$. Let $k$ be a number field and $\kbarre$ be
an algebraically closed extension of $k$. A class in the Brauer group of a projective smooth variety $X$ over $k$ is said to be \emph{algebraic} if
it belongs to the kernel of the restriction map $\Br(X) \rightarrow \Br(X_\kbarre)$, \emph{transcendental} otherwise; this property does not depend on
the choice of $\kbarre$. For any prime number $\ell$, the $\ell$-primary part of the Brauer group over $\C$ fits into an exact sequence
$$
\xymatrix{
0 \ar[r] & (\Ql/\Zl)^{b_2-\rho} \ar[r] & \Br(X_\C)\{\ell\} \ar[r] & H^3(X(\C),\Z)\{\ell\} \ar[r] & 0 \text{,}
}
$$
where $b_2$ and $\rho$ respectively denote the second Betti number and the Picard number of $X_\C$, and $M\{\ell\}$ denotes
the $\ell$-primary part of $M$.
Although this sequence does prove the non-triviality of $\Br(X_\C)$
in many cases, e.g. when $X$ is a $K3$ surface, transcendental classes are in general difficult to exhibit.

Almost all known instances of Brauer-Manin obstruction are thus explained by algebraic classes, the only exceptions being Harari's examples \cite{harari}
with conic bundles over $\P^2_\Q$.
Besides, in the particular case of pencils of curves of genus $1$, results on the Hasse principle have been obtained only under the assumption that
the $2$-primary part of the Brauer group be ``vertical'', and therefore algebraic (see \cite{css}, §4.7).
The rôle of transcendental elements in the Brauer-Manin obstruction thus seems worthy of investigation.
In this note we present an example of transcendental Brauer-Manin obstruction to weak approximation for an elliptic $K3$ surface over $\Q$,
where ``elliptic'' means that it possesses a fibration in curves of genus $1$, with a section, over $\P^1_\Q$. It should be noted
that the class of order $2$ which we will exhibit in $\Br(X_\C)$ enjoys the property of being divisible (because $H^3(X(\C),\Z)=0$ for a $K3$ surface),
which was not the case in Harari's examples.

\section{Preliminaries: $2$-descent and the Brauer group of an elliptic curve}

The subscript in $H^i_\et$ will be dropped, as we will only use étale cohomology.
If $G$ is an abelian group (resp. group scheme), $\tors{n}{G}$ will denote the $n$-torsion subgroup of $G$.
Let $k$ be a perfect field of characteristic different from~$2$.
The Hilbert symbol of a pair of elements $f, g \in k^\star$ will be denoted
$(f, g)$; it is the class of a quaternion algebra in $\tors{2}\Br(k)$.
When $X$ is a geometrically integral variety over $k$ and $L$ is an
extension of $k$, $L(X)$ will denote the function field of~$X_L$.
The canonical morphism $\Br(X) \rightarrow \Br(k(X))$ is injective if in addition $X$ is regular; this fact will be used without further mention.
Let $E$ be an elliptic curve over $k$ whose $2$-torsion points are rational.
Fix an isomorphism of $k$-group schemes $(\Z/2\Z)^2 \isoto \tors{2}E$. The kernel of the evaluation map at the zero section $\Br(E) \rightarrow \Br(k)$
will be denoted $\Br^0(E)$.

\begin{lemma}
\label{plem}
The group $\Br^0(E)$ is canonically isomorphic to $H^1(k,E)$.
\end{lemma}

\begin{proof}
Let us write the Leray spectral sequence for the structure morphism $f \colon E \rightarrow \Spec(k)$ and the étale
sheaf~$\Gm$. Since $f_\star \Gm = \Gm$, $R^1 f_\star \Gm = E \oplus \Z$ and $R^q f_\star \Gm = 0$ for $q > 1$ by Tsen's theorem, we get an exact sequence
$$
\xymatrix{
\Br(k) \ar[r] & \Br(E) \ar[r] & H^1(k, E) \ar[r] & H^3(k, \Gm) \ar[r] & H^3(E, \Gm) \text{.}
}
$$
The zero section induces retractions of $\Br(k) \rightarrow \Br(E)$ and of $H^3(k, \Gm) \rightarrow H^3(E, \Gm)$, hence
the lemma.
\end{proof}

The Kummer sequence
$$
\xymatrix@R=0ex{
0 \ar[r] & \tors{2}E \ar[r] & E \ar[r]^(0.46){z \mapsto 2z} & E \ar[r] & 0 \text{,}
}
$$
together with the previous lemma and the chosen isomorphism $(\Z/2\Z)^2 \isoto \tors{2}E$, yields the exact sequence
\begin{equation}
\label{sebr}
\xymatrix{
0 \ar[r] & E(k)/2E(k) \ar[r]^\delta & (k^\star/k^{\star 2})^2 \ar[r]^\gamma & \tors{2}\Br^0(E) \ar[r] & 0 \text{.}
}
\end{equation}

We shall need explicit descriptions of the maps $\delta$ and $\gamma$. First choose distinct $p, q \in k^\star$ such that
the Weierstrass equation
\begin{equation}
\label{eqw}
y^2 = x(x-p)(x-q)
\end{equation}
defines $E$ and the points $P=(p,0)$ and $Q=(q,0)$ are respectively sent
to $(1,0)$ and $(0,1)$ via $\tors{2}E \isoto (\Z/2\Z)^2$. It is well-known (see e.g. \cite{silv}, p. 281) that $\delta(M)=(x(M)-q, x(M)-p)$ for
$M \in E(k)$ if $M \not\in\tors{2}E(k)$, that $\delta(P)=(p-q, p(p-q))$ and that $\delta(Q)=(q(q-p), q-p)$.

\begin{prop}
\label{propdeltagamma}
Let $f, g \in k^\star$. The classes of the quaternion algebras $(x-p, f)$ and $(x-q, g) \in \Br(k(E))$ actually belong to $\Br^0(E)$,
and $\gamma(f,g)=(x-p,f)+(x-q,g)$.
\end{prop}

\begin{proof}
By symmetry, it is enough to prove that $\gamma(f,1)=(x-p,f)$ in $\Br(k(E))$.
Choose a separable closure $\kbarre$ of $k$ and let $G_k$ be its Galois group over $k$.
Likewise, choose a separable closure $\overline{k(E)}$ of $\kbarre(E)$ and let $G_{k(E)}$ be its Galois group over $k(E)$.
It follows from the Hochschild-Serre spectral sequence, Tsen's theorem and Hilbert's theorem~90
that the inflation map
$H^2(k, \kbarre(E)^\star) \rightarrow \Br(k(E))$ is an isomorphism.
Let $\rho \colon H^1(k,E) \rightarrow H^2(k,\kbarre(E)^\star/\kbarre^\star)$
denote the composition of the canonical isomorphism $H^1(k,E) \isoto H^1(k,\Pic(E_\kbarre))$ and the boundary of the exact sequence
$$
\xymatrix{
0 \ar[r] & \kbarre(E)^\star / \kbarre^\star \ar[r] & \Div(E_\kbarre) \ar[r] & \Pic(E_\kbarre) \ar[r] & 0 \text{.}
}
$$
As shown in the annexe of \cite{ctsan}, the diagram
$$
\xymatrix@R=2ex{
\Br(k) \ar[r] \ar@{=}[ddd] & \Br(E) \ar[r]^(0.47)\theta \ar|(.48){\bigcap}@{}[d] & H^1(k,E) \ar[ddd]^{-\rho} \\
& \Br(k(E)) \ar[dd]^(0.45)\wr \\ \\
\Br(k) \ar[r] & H^2(k, \kbarre(E)^\star) \ar[r] & H^2(k, \kbarre(E)^\star/\kbarre^\star)
}
$$
commutes, where $\theta$ denotes the map which stems from the Leray spectral sequence (see lemma \ref{plem}).
This enables us to carry out cocycle calculations for determining the image of $\gamma(f,1)$
in $H^2(k,\kbarre(E)^\star/\kbarre^\star)$. We shall use the standard cochain complexes.
Let $\chi_f \colon G_k \rightarrow \Z$ be the map with image in $\{0,1\}$ whose composition with the projection
$\Z \rightarrow \Z/2\Z$ is the quadratic character associated with $f \in k^\star/k^{\star 2} = H^1(G_k, \Z/2\Z)$.
The image of $(f,1)$ in $H^1(k, E)$ is represented by the $1$-cocycle $a \colon \sigma \mapsto \chi_f(\sigma) P$.
If $M \in E(k)$, let $[M]$ denote the corresponding divisor on $E_\kbarre$.
The $1$-cochain with values in $\Div(E_\kbarre)$ defined by $\sigma \mapsto \chi_f(\sigma) ([P]-[0])$ is a lifting of~$a$.
Its differential $(\sigma, \tau) \mapsto (\chi_f(\sigma) + \chi_f(\tau) - \chi_f(\sigma\tau))([P]-[0])$ is, as expected,
a $2$-cocycle with values in $\kbarre(E)^\star/\kbarre^\star$, which we may rewrite as $(\sigma, \tau) \mapsto (x-p)^{\chi_f(\sigma)\chi_f(\tau)}$;
it represents the image of $\gamma(f,1)$ in $H^2(k,\kbarre(E)^\star/\kbarre^\star)$. Since $x-p$ is invariant under $G_k$, the same formula
defines a $2$-cocycle on $G_k$ with values in $\kbarre(E)^\star$. We thus end up with a $2$-cocycle
$$
\begin{array}{rcl}
b \colon G_{k(E)} \times G_{k(E)} & \longrightarrow & \overline{k(E)}^\star \\
(\sigma, \tau) & \longmapsto & (x-p)^{\chi_f(\sigma)\chi_f(\tau)}
\end{array}
$$
which represents the image of $\gamma(f, 1)$ in $\Br(k(E))$, at least modulo $\Br(k)$, where $\chi_m$ now denotes the lifting with values
in $\{0,1\}$ of the quadratic character on $k(E)$ associated with $m \in k(E)^\star$. (Note that $k$ is separably closed in $k(E)$, so that
$G_k$ identifies with a quotient of $G_{k(E)}$.) Choose a square root $s$ of $x-p$ in $\overline{k(E)}$.
Dividing $b$ by the differential of the $1$-cochain $\sigma \mapsto s^{\chi_f(\sigma)}$ gives the $2$-cocycle
$(\sigma, \tau) \mapsto (-1)^{\chi_{x-p}(\sigma)\chi_f(\tau)}$, which does represent the image of the cup-product $(x-p) \cup f$ by
the composite map $H^1(k(E),\Z/2\Z)^{\otimes 2} \rightarrow H^2(k(E), \Z/2\Z) \rightarrow \Br(k(E))$.

We have now proved that $\gamma(f,1) = (x-p,f)$ in $\Br(k(E))/\Br(k)$, but the equality holds in $\Br(k(E))$ since $(x-p,f)=(y^2/(x-p)^3,f)$
evaluates to $0$ at the zero section.
\end{proof}

\section{An actual example}

The reader is referred to \cite{harari} for the definitions of weak approximation, Brauer-Manin obstruction, residue maps and unramified Brauer group.

Let $\Omega$ denote the set of places of $\Q$.
Define the polynomials $p, q \in \Q[t]$ by $p(t) = 3(t-1)^3(t+3)$ and $q(t)=p(-t)$. It will be useful
to notice that $p(t)-q(t) = 48 t$. Let $E$ be the elliptic curve over $\Q(t)$ defined by~(\ref{eqw}).
Denote by $\Erond$ its minimal proper regular model over $\P^1_\Q$ (see \cite{shaf}); it is a smooth surface over $\Q$ endowed
with a proper flat morphism $f \colon \Erond \rightarrow \P^1_\Q$ whose generic fibre is isomorphic to $E$. A geometric fibre of $f$
is either smooth or is a union of rational curves whose intersection numbers may be computed with Tate's algorithm~\cite{tate}. One finds
the following reduction types, in Kodaira's notation \cite{kod}: $I_2$ above $t=0$, $t=3$ and $t=-3$; $I_6$ above $t=1$, $t=-1$ and $t=\infty$; the
other fibres are smooth. Recall that a fibre of type $I_n$ has $n$ irreducible components $(C_i)_{1 \leq i \leq n}$, with $(C_i.C_{i+1})=1$,
$(C_1.C_n)=1$ and $(C_i.C_j)=0$ if $n-1>\abs{j-i}>1$.
Put $$A = \gamma(6t(t+1), 6t(t-1)) = (x-p, 6t(t+1)) + (x-q, 6t(t-1)) \in \Br(E) \text{.}$$

\begin{prop}
The class $A \in \Br(E)$ belongs to the subgroup $\Br(\Erond)$.
\end{prop}

\begin{proof}
Let $v$ be a discrete rank $1$ valuation on $\Q(\Erond)$ whose restriction to~$\Q$ is trivial,
and $\kappa$ be its residue field. We shall prove that $A$ has trivial residue at $v$.
Let us choose a uniformiser $\pi$ of $v$ and put
$\tilde{z}=z \pi^{-v(z)}$ for $z \in \Q(\Erond)^\star$. It will be convenient to denote by
$V \colon \Q(\Erond)^\star \rightarrow \Z \times \kappa^\star$ the group homomorphism
$z \mapsto (v(z), [\tilde{z}])$, where $[u]$ denotes the class in $\kappa$ of $u \in \Q(\Erond)$ if $v(u)=0$.
For $f, g \in \Q(\Erond)^\star$, the residue of the quaternion algebra $(f, g)$ at $v$ is given by the tame symbol formula
$$
\partial_v(f,g) = (-1)^{v(f)v(g)} \left[ \frac{f^{v(g)}}{g^{v(f)}} \right]
= (-1)^{v(f)v(g)} \left[\tilde{f}\right]^{v(g)} \left[\vphantom{\tilde{f}} \tilde{g}\right]^{v(f)}
 \in \kappa^\star / \kappa^{\star 2} \text{.}
$$
Note that it only depends on $V(f)$ and $V(g)$. Furthermore, if $V(f)$ is a double,
i.e. if $v(f)$ is even and $\tilde{f}$ is a square modulo $\pi$, then $\partial_v(f,g)=1$.
These remarks will be used implicitly throughout the proof.

\begin{lemma}
\label{lemmenr}
The class $(-p, 6t(t+1)) + (-q, 6t(t-1)) \in \Br(\Q(t))$ is unramified over $\P^1_\Q$.
\end{lemma}

\begin{proof}
The residue at a closed point of $\P^1_\Q$ other than $t=\alpha$ for $\alpha\in\{-3,-1,0,1,3,\infty\}$
is obviously trivial. It is straightforward to check that the remaining residues are also
trivial.
\end{proof}

Let us now turn to showing that $\partial_v(A)=1$. As $A$ is invariant under $t \mapsto -t$,
we may assume $v(p) \leq v(q)$. If $v(x) < v(p)$, then $V(x-p)=V(x-q)=V(x)$, from which
we deduce thanks to (\ref{eqw}) that $V(x-p)$ and $V(x-q)$ are doubles.
If $v(x) > v(q)$, then $V(x-p)=V(-p)$ and $V(x-q)=V(-q)$, hence the result by lemma~\ref{lemmenr}. From now on, we may and will therefore assume $v(p) \leq v(x) \leq v(q)$.

To begin with, suppose $v(p)<v(q)$. In this case, either $v(t-3)>0$ or $v(t+1)>0$.
If $v(x)=v(q)$, then $V(x-p)=V(-p)$, hence $\partial_v(A)=\partial_v(-q(x-q), 6t(t-1))$
by lemma \ref{lemmenr}; but with a look at (\ref{eqw}), one finds that both $v(-q(x-q))$ and $v(6t(t-1))$ are even.
Suppose now $v(x)<v(q)$. It follows from (\ref{eqw}) that $V(x-p)$ is a double,
hence $\partial_v(A)=\partial_v(x-q,6t(t-1))=\partial_v(x,6t(t-1))$.
If $v(x)$ is even or if $[6t(t-1)]$ is a square in~$\kappa$, which happens if $v(t-3)>0$, we get $\partial_v(A)=1$.
If on the other hand $v(t+1)>0$ and $v(x)$ is odd, then $[6t(t-1)]=12$, which (\ref{eqw})
shows to be a square in $\kappa$.

We are now left with the case $v(p)=v(q)=v(x)$. If $v(t)=0$, then $v(t-3)=v(t-1)=v(t+1)=v(t+3)=0$,
so $v(6t(t+1))=v(6t(t-1))=0$ and it suffices to prove that $v(x-p)$ and $v(x-q)$ are
even, which follows from (\ref{eqw}) and the equality $v(p)=v(x)=v(q)=v(p-q)=0$.
If $v(t)<0$, then $V(6t(t+1))=V(6t(t-1))$, so that $\partial_v(A)=\partial_v(x, 6t(t+1))$, which is trivial
since both $v(x)=v(p)=4v(t)$ and $v(6t(t+1))$ are even.
Suppose finally that $v(t)>0$. If $v(x-p)<v(t)$, then $V(x-p)=V(x-q)$ since $v(p-q)=v(t)$,
and $\partial_v(A)=\partial_v(x-p,(t+1)(t-1))=\partial_v(x-p,-1)$; if $v(x-p)=0$, the residue is obviously trivial,
and if $v(x-p)>0$, which means that $[\tilde{x}]=[\tilde{p}]=-9$, (\ref{eqw}) shows that $-1$ is a square in $\kappa$.
We therefore assume $v(x-p)\geq v(t)$, which still leads to $[\tilde{x}]=[\tilde{p}]=-9$.
As $v(p-q)=v(t)$, at least one of $v(x-p)$ and $v(x-q)$ is equal to $v(t)$. In either case, (\ref{eqw}) implies that $v(x-p)+v(t)$ is even,
so $(-9)^{v(t)}(-1)^{v(x-p)}$ is a square, hence
$\partial_v(A)=\partial_v(x,6t(t-1))+\partial_v(x-p,(t+1)(t-1))$ is trivial.
\end{proof}

We shall now prove the following.

\begin{theorem}
The class $A \in \Br(\Erond)$ is transcendental and yields a Brauer-Manin obstruction to weak approximation on the projective smooth surface $\Erond$ over $\Q$.
\end{theorem}

\begin{proof}
Let us first deal with the second part of the assertion. A glance at equation (\ref{eqw}) shows that $\Erond$ has a $\Q_2$-point
$M_2$ with coordinates $x=1$ and $t=2$. (Indeed, this equation defines an affine surface over $\Q$ endowed with a morphism
to $\P^1_\Q$ whose smooth locus identifies with an open subset of $\Erond$.) Using the formula given in~\cite{serre}, Ch.~XIV, §4,
one easily checks that $A(M_2)$ is non-trivial. Now choose $N \in \Erond(\Q)$ in the image of the zero section and let
$M_v \in \Erond(\Q_v)$ be equal to $N$ for any $v \in \Omega\setminus\{2\}$.
This defines an adelic point $(M_v)_{v \in \Omega}$. The class $A(N) \in \Br(\Q)$ is trivial since $A \in \Br^0(E)$;
consequently, the evaluation of $A$ at $(M_v)_{v \in \Omega}$ is non-trivial, which is an obstruction to weak approximation.

It remains to be shown that $A$ is transcendental. The exact sequence (\ref{sebr}) reduces this to the computation
of $E(\C(t))/2E(\C(t))$.

\begin{lemma}
The surface $\Erond$ is a $K3$ surface.
\end{lemma}

\begin{proof}
The topological Euler-Poincaré characteristic $e(\Erond_\C)$ of $\Erond_\C$ can be expressed in terms of that of the fibres and that of the base
(\cite{bpv}, p.~97, prop.~11.4), which leads to $e(\Erond_\C)=24$. Let $\chi(\Orond_\Erond)$ denote the Euler-Poincaré characteristic of the
coherent sheaf $\Orond_\Erond$. The canonical bundle $\Krond_\Erond$ of $\Erond$ is simply $f^\star \Orond(\chi(\Orond_\Erond)-2)$ (see
\cite{bpv}, p.~162, cor.~12.3); in particular it has self-intersection $0$, hence $\chi(\Orond_\Erond)=2$ by Noether's formula. We have now
proved the triviality of $\Krond_\Erond$. That $H^1(\Erond, \Orond_\Erond)=0$ follows from $\chi(\Orond_\Erond)=2$ and Serre duality.
\end{proof}

\begin{lemma}
The elliptic curve $E$ has Mordell-Weil rank $0$ over $\C(t)$.
\end{lemma}

\begin{proof}
Let $\rho(\Erond_\C)$ be the Picard number of $\Erond_\C$ and $R$ be the subgroup of the Néron-Severi group $\NS(\Erond_\C)$ spanned by the zero
section and the irreducible components of the fibres. As follows from the output of Tate's algorithm, $R$ has rank $20$.
On the other hand, $\rho(\Erond_\C) \leq 20$ since $\Erond$ is a $K3$ surface. The Shioda-Tate formula $$\rho(\Erond_\C) = \rk(E(\C(t))) + \rk(R)$$
thus yields the result.
\end{proof}

This lemma shows that the $\F_2$-vector space $E(\C(t))/2E(\C(t))$ has dimension $2$. Now the classes $\delta(P)=(t,t(t-1)(t+3))$
and $\delta(Q)=(t(t+1)(t-3), t)$ are independent over $\F_2$, hence span the whole kernel of~$\gamma$. On the other hand $(t(t+1), t(t-1))$
is evidently not a combination of $\delta(P)$ and $\delta(Q)$, so that $A$ has non-zero image in $\Br(\C(\Erond))$ and is therefore transcendental.
\end{proof}

\begin{remark} \normalfont
It is actually true that $A(M)=0$ in $\Br(\Q)$ for \emph{all} $M \in \Erond(\Q)$. This is a consequence of the global reciprocity law
and the fact that $A$ vanishes on $\Erond(\Q_v)$ for all $v \in \Omega\setminus\{2\}$, which can be checked by a tedious computation.
\end{remark}

\begin{remark} \normalfont
It is possible to determine $\tors{2}\Br(\Erond)$ completely if one is willing to compute explicit equations for~$\Erond$. This involves
blowing up the singular surface given by equation (\ref{eqw}) a sufficient number of times. Alternatively, one may observe that all fibres
have type $I_n$ (in other words, $\Erond \rightarrow \P^1_\Q$ is semi-stable), and then use the equations given by Néron in this case
in \cite{neron}, §III.
Either way one finds that $\tors{2}\Br(\Erond)$ is spanned by $A$ modulo $\tors{2}\Br(\Q)$ after writing out all possible residues of a general
class $\gamma(f,g)$.
On the other hand, the $2$-torsion subgroup of the Brauer group of a complex $K3$ surface with Picard number $20$ has
rank $2$ over $\F_2$, so $\tors{2}\Br(\Erond_\C)$ is strictly larger than $\tors{2}\Br(\Erond)/\tors{2}\Br(\Q)$.
It turns out that $\tors{2}\Br(\Erond_\C)$ is spanned by $A$ and the class of the quaternion algebra $(x, t)$, which unexpectedly belongs
to $\Br(\Q(\Erond))$ and only gets unramified after extension of scalars to $\Q(\sqrt{-1}, \sqrt{3})$.
\end{remark}

\begin{remark} \normalfont
In the semi-stable case, a computer program was written to carry out the calculations alluded to in the previous paragraph, as they
often get quite lengthy. Its source code is available on request.
\end{remark}

\section*{Acknowledgements}

The author is most grateful to J-L. Colliot-Thélène for sharing unpublished notes on the topic (which contain in particular
the statement of proposition \ref{propdeltagamma}), and would also like
to thank him for his encouragements and many helpful conversations during the course of this research.

\bibliographystyle{amsplain}
\bibliography{transcendental}

\providecommand{\bysame}{\leavevmode\hbox to3em{\hrulefill}\thinspace}
\providecommand{\MR}{\relax\ifhmode\unskip\space\fi MR }
\providecommand{\MRhref}[2]{%
  \href{http://www.ams.org/mathscinet-getitem?mr=#1}{#2}
}
\providecommand{\href}[2]{#2}
\begin{thebibliography}{10}

\bibitem{bpv}
W.~Barth, C.~Peters, and A.~Van de~Ven, \emph{Compact complex surfaces},
  Ergebnisse der Mathematik und ihrer Grenzgebiete (3), vol.~4,
  Springer-Verlag, Berlin, 1984.

\bibitem{ctsan}
J-L. Colliot-Th{\'e}l{\`e}ne and J-J. Sansuc, \emph{La {$R$}-\'equivalence sur
  les tores}, Ann. Sci. \'Ecole Norm. Sup. (4) \textbf{10} (1977), no.~2,
  175--229.

\bibitem{css}
J-L. Colliot-Th{\'e}l{\`e}ne, A.~N. Skorobogatov, and Sir~Peter
  Swinnerton-Dyer, \emph{Hasse principle for pencils of curves of genus one
  whose {J}acobians have rational {$2$}-division points}, Invent. math.
  \textbf{134} (1998), no.~3, 579--650.

\bibitem{harari}
D.~Harari, \emph{Obstructions de {M}anin transcendantes}, Number theory (Paris,
  1993--1994), London Math. Soc. Lecture Note Ser., vol. 235, Cambridge Univ.
  Press, Cambridge, 1996, pp.~75--87.

\bibitem{kod}
K.~Kodaira, \emph{On compact analytic surfaces {II}}, Ann. of Math. (2)
  \textbf{77} (1963), 563--626.

\bibitem{neron}
A.~N{\'e}ron, \emph{Mod\`eles minimaux des vari\'et\'es ab\'eliennes sur les
  corps locaux et globaux}, Inst. Hautes \'Etudes Sci. Publ. Math. No.~
  \textbf{21} (1964).

\bibitem{serre}
J-P. Serre, \emph{Corps locaux}, Hermann, Paris, 1968.

\bibitem{shaf}
I.~R. Shafarevich, \emph{Lectures on minimal models and birational
  transformations of two dimensional schemes}, Notes by C.~P. Ramanujam, Tata
  Institute of Fundamental Research Lectures on Mathematics and Physics,
  No.~37, Tata Institute of Fundamental Research, Bombay, 1966.

\bibitem{silv}
J.~H. Silverman, \emph{The arithmetic of elliptic curves}, Graduate Texts in
  Mathematics, vol. 106, Springer-Verlag, New York, 1992.

\bibitem{tate}
J.~Tate, \emph{Algorithm for determining the type of a singular fiber in an
  elliptic pencil}, Modular functions of one variable, IV (Proc. Internat.
  Summer School, Univ. Antwerp, Antwerp, 1972), Springer, Berlin, 1975,
  pp.~33--52, Lecture Notes in Math., Vol. 476.

\end{thebibliography}

\end{document}